\documentclass[11pt,leqno]{amsart}
\usepackage{amssymb,amsmath,amsthm,amsfonts}
\usepackage {latexsym}
\usepackage{bbm} 

\allowdisplaybreaks
\newcommand{\nc}{\newcommand}
\nc{\les}{\lesssim}
\nc{\nit}{\noindent}
\nc{\nn}{\nonumber}
\nc{\D}{\partial}
\nc{\diff}[2]{\frac{d #1}{d #2}}
\nc{\diffn}[3]{\frac{d^{#3} #1}{d {#2}^{#3}}}
\nc{\pdiff}[2]{\frac{\partial #1}{\partial #2}}
\nc{\pdiffn}[3]{\frac{\partial^{#3} #1}{\partial{#2}^{#3}}}
\nc{\abs}[1] {\lvert #1 \rvert}
\nc{\cAc}{{\cal A}_c}
\nc{\cE}{{\cal E}}
\nc{\cF}{{\cal F}}
\nc{\cP}{{\cal P}}
\nc{\cV}{{\cal V}}
\nc{\cQ}{{\cal Q}}
\nc{\cGin}{{\cal G}_{\rm in}}
\nc{\cGout}{{\cal G}_{\rm out}}
\nc{\cO}{{\cal O}}
\nc{\Lav}{{\cal L}_{\rm av}}
\nc{\cL}{{\cal L}}
\nc{\cB}{{\cal B}}
\nc{\cZ}{{\cal Z}}
\nc{\cR}{{\cal R}}
\nc{\cT}{{\cal T}}
\nc{\cY}{{\cal Y}}
\nc{\cX}{{\cal X}}
\nc{\cXT}{{{\cal X}(T)}}
\nc{\cBT}{{{\cal B}(T)}}
\nc{\vD}{{\vec \mathcal{D}}}
\nc{\efield}{\mathcal{E}}
\nc{\vE}{{\vec \efield}}
\nc{\vB}{{\vec \mathcal{B}}}
\nc{\vH}{{\vec \mathcal{H}}}
\nc{\cH}{\mathcal H}
\nc{\ty}{{\tilde y}}
\nc{\tu}{{\tilde u}}
\nc{\tV}{{\tilde V}}
\nc{\Pc}{{\bf P_c}}
\nc{\bx}{{\bf x}}
\nc{\bX}{{\bf X}}
\nc{\bXYZ}{{\bf XYZ}}
\nc{\bY}{{\bf Y}}
\nc{\bF}{{\bf F}}
\nc{\bS}{{\bf S}}
\nc{\dV}{{\delta V}}
\nc{\dE}{{\delta E}}
\nc{\TT}{{\Theta}}
\nc{\dPsi}{{\delta\Psi}}
\nc{\order}{{\cal O}}
\nc{\Rout}{R_{\rm out}}
\nc{\eplus}{e_+}
\nc{\eminus}{e_-}
\nc{\epm}{e_\pm}
\nc{\eps}{\varepsilon}
\nc{\vnabla}{{\vec\nabla}}
\nc{\G}{\Gamma}
\nc{\w}{\omega}
\nc{\mh}{h}
\nc{\mg}{g}
\nc{\vphi}{\varphi}
\nc{\tlambda}{\tilde\lambda}
\nc{\be}{\begin{equation}}
\nc{\ee}{\end{equation}}
\nc{\ba}{\begin{eqnarray}}
\nc{\ea}{\end{eqnarray}}

\nc{\g}{\gamma}
\nc{\ol}{\overline}

\newtheorem{theorem}{Theorem}[section]
\newtheorem{lemma}[theorem]{Lemma}
\newtheorem{prop}[theorem]{Proposition}
\newtheorem{corollary}[theorem]{Corollary}
\newtheorem{defin}[theorem]{Definition}

\nc{\pT}{\partial_T}
\nc{\pz}{\partial_z}
\nc{\pt}{\partial_t}
\nc{\la}{\langle}
\nc{\ra}{\rangle}
\nc{\infint}{\int_{-\infty}^{\infty}}
\nc{\halfwidth}{6.5cm}
\nc{\figwidth}{10cm}
\newcommand{\f}{\frac}

\nc{\nlayers}{L} \nc{\nsectors}{M}
\nc{\indicator}{\mathbf{1}}
\nc{\Rhole}{R_{\rm hole}}
\nc{\Rring}{R_{\rm ring}}
\nc{\neff}{n_{\rm eff}}
\nc{\Frem}{F_{\rm rem}}
\nc{\R}{\mathbb R}
\nc{\Z}{\mathbb Z}
\nc{\C}{\mathbb C}
\nc{\DD}{\Delta}
\nc{\cD}{\mathcal D}
\nc{\lnorm}{\left\|}
\nc{\rnorm}{\right\|}
\nc{\rnormp}{\right\|_{\ell^{p,\eps}}}
\nc{\rar}{\rightarrow}
\date{\today}
\begin{document}

\begin{abstract}
We consider the  non-selfadjoint operator
\[ \cH =
\left[\begin{array}{cc} -\Delta + \mu-V_1 & -V_2\\
V_2 & \Delta - \mu + V_1
\end{array}
\right]
\]
where $\mu>0$ and $V_1,V_2$ are real-valued decaying potentials. Such operators arise when linearizing a focusing NLS equation around
a standing wave.
Under natural spectral assumptions we obtain   $L^1(\R^2)\times L^1(\R^2)\to
L^\infty(\R^2)\times L^\infty(\R^2)$ dispersive decay estimates for   the evolution $e^{it\cH}P_{ac}$.
We also obtain the following weighted estimate
$$
\|w^{-1} e^{it\cH}P_{ac}f\|_{L^\infty(\R^2)\times L^\infty(\R^2)}\les \f1{|t|\log^2(|t|)} \|w f\|_{L^1(\R^2)\times L^1(\R^2)},\,\,\,\,\,\,\,\, |t| >2,
$$
with $w(x)=\log^2(2+|x|)$.
\end{abstract}

\title[Dispersive Estimates for Matrix Schr\"odinger Operators]{\textit{Dispersive estimates for
matrix Schr\"{o}dinger operators in dimension two}}

\author[M.~B. Erdo\smash{\u{g}}an, W.~R. Green]{M. Burak Erdo\smash{\u{g}}an and William~R. Green}

\address{Department of Mathematics \\
University of Illinois \\
Urbana, IL 61801, U.S.A.}
\email{berdogan@math.uiuc.edu}
\address{Department of Mathematics\\
Rose-Hulman Institute of Technology\\
Terre Haute, IN 47803, U.S.A.}
\email{green@rose-hulman.edu}

\maketitle
\section{Introduction}
The free Schr\"{o}dinger evolution on $\mathbb R^d$,
\be\label{freedecay}
 e^{-it\Delta}f (x)= C_d \frac{1}{t^{d/2}}\int_{\R^d} e^{-i|x-y|^2/4t}f(y) dy,
\ee
satisfies the dispersive estimate
$$\| e^{-it\Delta}f \|_\infty \lesssim \f1{|t|^{d/2}} \| f\|_1.$$
In recent years many authors (see, e.g.,  \cite{JSS,RodSch,GS,Sc2,GV,Yaj,FY,CCV,EG1,Gr,BG},
 and the survey article \cite{Scs}) worked on the problem of extending this bound to the perturbed Schr\"{o}dinger operator $H=-\Delta+V$, where $V$ is a real-valued potential with sufficient decay at infinity (some smoothness is required for $d>3$). Since the perturbed operator may have negative point spectrum one needs to consider $e^{itH}P_{ac}(H)$, where $P_{ac}(H)$ is the orthogonal projection onto the absolutely continuous subspace of $L^2(\R^d)$.
Another common  assumption is that zero is a regular point of the spectrum of $H$.

We note that the $L^1\to L^\infty$ estimates were preceded by somehow weaker estimates on  weighted $L^2$ spaces,  see, e.g.,
\cite{Rauch,JenKat,Mur}.

Although the $L^1\to L^\infty$ estimates are very well studied in the three dimensional case, there are not   many results   in dimension two.
In \cite{Sc2}, Schlag proved that
\begin{align}\label{schresult}
\|e^{itH}P_{ac}\|_{L^1(\R^2)\to L^\infty(\R^2)}\lesssim |t|^{-1}
\end{align}
under the decay assumption $|V|\lesssim \langle x\rangle^{-3-}$ and the assumption that zero is a regular point of the spectrum.  For the case when zero is not regular, see \cite{EG}. Yajima, \cite{Yaj2}, established that the wave operators are bounded on $L^p(\R^2)$ for $1<p<\infty$ if zero is regular.  The hypotheses
on the potential $V$ were
relaxed slightly in \cite{JenYaj}.

Note that the decay rate in \eqref{freedecay} is not integrable at infinity for $d=1,2$. However, in dimensions $d=1$ and $d=2$,  zero is not a regular point of the spectrum of the Laplacian (the constant function is a resonance). Therefore, for the perturbed operator $-\Delta+V$, one may expect to have a faster dispersive decay at infinity if zero is regular. Indeed, in \cite[Theorem 7.6]{Mur}, Murata proved that if zero is a regular point of the spectrum, then for $|t|>2$
\begin{align*}
\| w_1^{-1}e^{itH}P_{ac}(H) f\|_{L^2(\R^1)} &\les |t|^{-\f32} \|w_1 f\|_{L^2(\R^1)},\\
\| w_2^{-1}e^{itH}P_{ac}(H) f\|_{L^2(\R^2)} &\les |t|^{-1}(\log |t|)^{-2}\|w_2 f\|_{L^2(\R^2)}.
\end{align*}
 Here $w_1$ and $w_2$ are weight functions growing at a polynomial rate at infinity. It is also assumed that the potential decays at a polynomial rate at infinity (for $d=2$, it suffices to assume that $w_2(x)=\la x\ra^{-3-}$ and  $|V(x)|\les \la x\ra^{-6-}$ where $\la x\ra:= (1+|x|^2)^\f12$).
This type of estimates are very useful in the study of nonlinear asymptotic stability of (multi) solitons in lower dimensions since the dispersive decay rate  in time is integrable at infinity (see, e.g., \cite{scns,KZ}). Also see \cite{SW1,BP,PW,Wed} for other applications of weighted dispersive estimates to nonlinear PDEs.

In \cite{Scs},
Schlag extended Murata's result for $d=1$ to the $L^1\to L^\infty$ setting (also see \cite{Gtrans} for an improved result). In \cite{EG2}, the authors obtained an analogous estimate for $d=2$:
If zero is a regular point of the spectrum of $H$, then
\begin{align}\label{scalarcase}
	\|w^{-1} e^{itH}P_{ac}(H) f\|_{L^\infty(\R^2)}\les \frac{1}{|t|\log^2(|t|)} \|w f\|_{L^1(\R^2)}, \,\,\,\,\,\,\,\, |t|>2,
\end{align}
with $w(x)=\log^2(2+|x|)$ provided $ |  V(x)|\les \la x\ra^{-3-} $.

In this paper we extend Schlag's result \eqref{schresult} and our result \eqref{scalarcase} for the $2d$ scalar Schr\"odinger operator to  the $2d$ non self-adjoint matrix Schr\"odinger operator
\begin{align}\label{matrix defn}
  \mathcal H = \mathcal H_0+V=\left[ \begin{array}{cc} -\Delta+\mu & 0 \\ 0 & \Delta-\mu\end{array}\right]
  +\left[ \begin{array}{ll} -V_1 & -V_2 \\ V_2 & V_1 \end{array} \right],\,\,\,\,\,\mu>0.
\end{align}
Such operators appear naturally as linearizations of a nonlinear
Schr\"odinger equation around a standing wave. Dispersive estimates in the context of such
linearizations were obtained in  \cite{Cuc, RSS1, Sch1, ES, CT,Marz, Gr}.

Note that, by Weyl's criterion and the decay assumption on $V_1$ and $V_2$ below,  the essential spectrum of $\mathcal H$ is
$(-\infty,-\mu]\cup[\mu,\infty)$.
Recall the Pauli spin matrix
\begin{align*}
	\sigma_3=\left[\begin{array}{cc} 1 & 0\\ 0 & -1\end{array}\right].
\end{align*}
As in \cite{ES}, we make the following assumptions:

\noindent  A1) $-\sigma_3 V$ is a positive matrix,\\
  A2) $L_-=-\Delta+\mu-V_1+V_2\geq 0$, \\ 
  A3) $|V_1(x)|+|V_2(x)|\lesssim \langle x\rangle^{-\beta}$ for some $\beta>3$,\\
  A4) There are no embedded eigenvalues in $(-\infty, -\mu)\cup (\mu,\infty)$.\\
  A5) The threshold points $\pm \mu$ are regular points of the spectrum of $\cH$, see Definition~\ref{def:regular} below.

As it was noted in \cite{ES}, the first three assumptions are known to hold in the case of the linearized nonlinear Schr\"{o}dinger equation (NLS) when
the linearization is performed about a positive ground state standing wave.  Let, for some $\mu>0$, $\psi(t,x)=e^{it\mu}\phi(x)$ be 
  a standing wave solution of the NLS 
\begin{align}\label{NLS}
	i\partial_t \psi +\Delta \psi+ |\psi|^{2\gamma} \psi=0,
\end{align}
for some $\gamma>0$.  Here $\phi$ is a ground state:
\begin{align*}
	&\mu \phi-\Delta \phi=\phi^{2\gamma+1}, \qquad\phi>0.
\end{align*}
It was proven, see for example \cite{Strauss,BerLi}, that the
ground state solutions exist and further are
positive, smooth, radial,
exponentially decaying functions, see \cite{ES} for further discussion.

Linearizing about this ground state yields the matrix Schr\"{o}dinger equation with potentials
$V_1=(\gamma+1)\phi^{2\gamma}$ and $V_2=\gamma \phi^{2\gamma}$. Note that $V_1>0$ and $V_1>|V_2|$, which is the same as Assumption~A1).
Assumption~A2) holds because of the exponential decay of $\phi$. Also note that $L_- = -\Delta+\mu - \phi^{2\gamma}\geq 0$, since 
$L_-\phi=0$ and   $\phi>0$.
The assumption~A4) seems to hold for this example in the three-dimensional case as evidenced in the numerical studies \cite{DemSch,MS}.

The assumption~A5) is also standard, since the behavior of the resolvent
near the thresholds, $\pm\mu$, determine the decay rate (see \cite{Scs,EG} for the scalar case). We do not consider the case when the thresholds
$\pm \mu$ are not regular in this paper.

 Our main result is the following
\begin{theorem}\label{thm:main}
Under the assumptions A1) -- A5), we have
\be\label{nonweigted}
\| e^{it\cH}P_{ac}f\|_{L^\infty(\R^2)\times L^\infty(\R^2)}\les \f1{|t| } \| f\|_{L^1(\R^2)\times L^1(\R^2)},
\ee
and
\be\label{weighted}
\|w^{-1} e^{it\cH}P_{ac}f\|_{L^\infty(\R^2)\times L^\infty(\R^2)}\les \f1{|t|\log^2(|t|)} \|w f\|_{L^1(\R^2)\times L^1(\R^2)},\,\,\,\,\,\,\,\, |t| >2,
\ee
where $w(x)=\log^2(2+|x|)$.
\end{theorem}

In an attempt for brevity of this paper, we will try to use the lemmas from the scalar results \cite{Scs,EG2} as much as possible.
The most important step in the proof of Theorem~\ref{thm:main} is the analysis of the resolvent around the thresholds $\pm\mu$.
Once we obtain these expansions, it will be possible to relate and/or reduce the proof to the scalar case for most of the terms.

In addition to being of mathematical interest, we wish to note
that such estimates above are of use in the study of non-linear
PDEs, particularly the NLS.  Much work studying the
NLS linearizes the equation about groundstate or
standing wave solutions.
We note, in particular,
\cite{PW,GJLS,Wed,Cuc,KZ,Miz3,CucMiz,CT,Sch1,Becman2} and the
survey paper \cite{scns}.

\section{Spectral Theory of  Matrix Schr\"{o}dinger  Operators}\label{ST sect}

Consider the matrix Schr\"{o}dinger operator, given in \eqref{matrix defn},
on $L^2(\R^n)\times L^2(\R^n)$.  Here $\mu>0$ and $V_1$, $V_2$ are real-valued decaying potentials.  It follows from Weyl's criterion
 that the essential spectrum of $\mathcal H$ is $(-\infty,-\mu]\cup [\mu,\infty)$,
see e.g. \cite{HS, RS1}.

For the spectral theory of the matrix Schr\"{o}dinger operator, we refer the reader to
\cite{ES}.  Since most of the proofs presented in \cite{ES} are independent of dimension,  we cite the
results without proof.  Further spectral theory for the three dimensional
case can be found in \cite{AzSimp,CHS}.

\begin{lemma} \cite[Lemma~3]{ES} \label{spectrum lemma}
Let $\beta>0$ be arbitrary in A2), then the essential spectrum of $\mathcal H$ equals
  $(-\infty, -\mu]\cup[\mu,\infty)$.  Moreover $spec(\mathcal H)=-spec(\mathcal H)=\overline{spec(\mathcal H)}
  =spec(\mathcal H^*)$, and $spec(\mathcal H)\subset \R\cup i\R$.  The discrete spectrum of $\mathcal H$ consists
  of eigenvalues $\{z_j\}_{j=1}^{N}$, $0\leq N\leq\infty$, of finite multiplicity.  For each $z_j\neq 0$, the
  algebraic and geometric multiplicity coincide and $Ran(\mathcal H-z_j)$ is closed.  The zero eigenvalue
  has finite algebraic multiplicity, i.e., the generalized eigenspace $\cup_{k=1}^{\infty} ker(\mathcal H^k)$
  has finite dimension.  In fact, there is a finite $m\geq 1$ so that $ker(\mathcal H^k)=ker(\mathcal H^{k+1})$
  for all $k\geq m$.

\end{lemma}

As in the scalar case, see \cite{GS, EG} etc., the proofs will hinge on the limiting absorption principle of
Agmon \cite{agmon}.  We now state such a result from \cite{ES} for $(\mathcal H-z)^{-1}$ for $|z|>\mu$. Define the space
\begin{align*}
  X_{\sigma}:=L^{2,\sigma}(\R^2)\times L^{2,\sigma}(\R^2),\,\,\,\text{ where } L^{2,\sigma}(\R^n)=\{f:\langle x\rangle^{\sigma} f\in L^2(\R^n)\}.
\end{align*}
It is clear that $X_{\sigma}^*=X_{-\sigma}$.  The limiting absorption principle of Agmon is   formulated below.

\begin{prop}\label{LAP2}

  Let $\beta>1$, $\sigma>\frac{1}{2}$ and fix an arbitrary $\lambda_0>\mu$.  Then
  \begin{align}\label{LAP eqn}
    \sup_{|\lambda|\geq \lambda_0, \epsilon\geq 0} |\lambda|^{\frac{1}{2}} \|(\mathcal H-(\lambda\pm i\epsilon))^{-1}\|<\infty
  \end{align}
  where the norm is in $X_{\sigma}\to X_{-\sigma}$.

\end{prop}

\begin{proof}
  See Lemma~4, Proposition~5 and Corollary~6 of \cite{ES}.
\end{proof}

Using the explicit form of the free resolvent $\mathcal R_0 (\lambda)=(\cH_0-\lambda)^{-1}$, $\lambda \not \in (-\infty,\-\mu]\cup[\mu,\infty)$ (see the next section), one can define the limiting operators ($X_{\sigma}\to X_{-\sigma}$)
$$
\mathcal R_0^\pm(\lambda):= \lim_{\epsilon\to 0^+} \mathcal R_0(\lambda\pm i\epsilon), \,\,\,\,\, \lambda\in (-\infty -\mu)\cup(\mu,\infty).
$$
By Proposition~\ref{LAP2}, for fixed $\lambda_0>\mu$,
\be\label{free lap}
\sup_{|\lambda|\geq \lambda_0} |\lambda|^{\frac{1}{2}} \|\mathcal R_0^\pm(\lambda)\|_{X_{\sigma}\to X_{-\sigma}}<\infty.
\ee
One also have the derivative bounds  
\begin{align}\label{lap}
	\sup_{|\lambda|>\lambda_0} \|\partial_\lambda^k \mathcal R_0^{\pm}(\lambda)\|_{X_{\sigma}\to X_{-\sigma}}<\infty,
\end{align}
for $k=0,1,2$ with $\sigma>k+\frac{1}{2}$.

By resolvent identity, one can also define the operators
$$\mathcal R_V^\pm (\lambda):= \lim_{\epsilon\to 0^+} \mathcal R_V(\lambda\pm i\epsilon)=\lim_{\epsilon\to 0^+}  (\cH-(\lambda\pm i\epsilon))^{-1}$$
 for $\lambda\in (-\infty -\mu)\cup(\mu,\infty)$ and they satisfy \eqref{free lap} and \eqref{lap}, see \cite{ES} for details.

We also need the following spectral representation
of the solution operator, see \cite[Lemma~12]{ES}.

\begin{lemma}

	Under the assumptions A1)-A5), we have the representation
	\begin{align}
		e^{it\mathcal H}&=e^{it\mathcal H}P_{ac} +\sum_j e^{it\mathcal H}P_{\lambda_j},\,\,\,\,\,\text{ where }\nn\\
        \label{spec_repr}
            e^{it\mathcal H}P_{ac} &=\frac{1}{2\pi i}\int_{|\lambda|>\mu}
		e^{it\lambda}[\mathcal R_V^+(\lambda)-\mathcal R_V^-(\lambda)]\,d\lambda,
	\end{align}
	and the sum is over the discrete spectrum $\{\lambda_j\}_j$
	and $P_{\lambda_j}$ is the Riesz projection corresponding
	to the eigenvalue $\lambda_j$.
\end{lemma}

This representation is to be understood in the weak sense.
That is for $\psi, \phi$ in $W^{2,2}\times W^{2,2}\cap X_{1+}$
we have
\begin{align}
	\la e^{it\cH}\phi,\psi\ra
	= \frac{1}{2\pi i}\int_{|\lambda|>\mu}
	e^{it\lambda}\la[\mathcal R_V^+(\lambda)- \mathcal R_V^-(\lambda)]\phi,\psi\ra\,d\lambda +\sum_j \la e^{it\cH}P_{\lambda_j}\phi,\psi \ra.
	\label{spec rep L2}
\end{align}

In light of this representation, the first claim of Theorem~\ref{thm:main} follows from the following theorem. Let $\chi$ be a smooth cutoff for the interval $[-1,1]$.
\begin{theorem}\label{thm:nonweighted}
Under the assumptions A1) -- A5), we have, for any $t\in \R$,
\be\label{nonweigted1}
\sup_{x,y\in\R^2, L>1} \Big|\int_{|\lambda|>\mu} e^{it\lambda}\chi(\lambda/L)[\mathcal R_V^+(\lambda)-\mathcal R_V^-(\lambda)](x,y)\,d\lambda\Big|\les \frac{1}{|t|}.
\ee
\end{theorem}

The second claim of Theorem~\ref{thm:main} follows from the following theorem and Theorem~\ref{thm:nonweighted} by a simple interpolation (see \cite{EG2})
\begin{theorem}\label{thm:weighted}
Under the assumptions A1) -- A5), we have, for $|t|>2$,
\be\label{weighted1}
\sup_{ L>1} \Big|\int_{|\lambda|>\mu} e^{it\lambda}\chi(\lambda/L)[\mathcal R_V^+(\lambda)-\mathcal R_V^-(\lambda)](x,y)\,d\lambda\Big|\les \frac{ \sqrt{w(x)w(y)}}{|t| \log^2(|t|)}+\frac{\la x\ra^{3/2} \la y\ra^{3/2}}{|t|^{1+\alpha}},
\ee
where $w(x)=\log^2(2+|x|)$ and $0<\alpha<\frac{\beta-3}{2}$.
\end{theorem}

\section{Properties of the Free Resolvent} \label{sec:exp}

For $z\not\in(-\infty,-\mu]\cup[\mu,\infty)$,   the free resolvent is an integral operator
\begin{align} \label{r0lambda}
    \mathcal R_0(z) = (\mathcal{H}_0-z)^{-1}=\left[\begin{array}{cc} R_0(z-\mu) & 0\\ 0 & -R_0(-z-\mu)
    \end{array}\right],
\end{align}
where $R_0$ denoting the scalar free
resolvent operators,  $R_0(z)=(-\Delta-z)^{-1}$, $z\in\C\backslash [0,\infty)$.
We first recall some properties of   $R_0(z)$.

To simplify the formulas, we   use the notation
$$
f=\widetilde O(g)
$$
to denote
$$
\frac{d^j}{d\lambda^j} f = O\big(\frac{d^j}{d\lambda^j} g\big),\,\,\,\,\,j=0,1,2,3,...
$$
If the derivative bounds hold only for the first $k$ derivatives we  write $f=\widetilde O_k (g)$.

Recall that
\begin{align}\label{R0 def}
	R_0(z)(x,y)=\frac{i}{4} H_0^{+}(z^{1/2}|x-y|),
\end{align}
where $\Im(z^{1/2})>0$ and $H_0^{\pm}$ are modified Hankel functions
$$
H_0^{\pm}(z)=J_0(z)\pm iY_0(z).
$$
From the series expansions for the Bessel functions, see \cite{AS},  we have, as $z\to 0$,
\begin{align}
	J_0(z)&=1-\frac{1}{4}z^2+\frac{1}{64}z^4+\widetilde O_6(z^6),\label{J0 def}\\
	Y_0(z)&=\frac{2}{\pi}(\log(z/2)+\gamma)J_0(z)+\frac{2}{\pi}\bigg(\frac{1}{4}z^2 +\widetilde O_4(z^4)\bigg)\nn \\
&=\frac2\pi \log(z/2)+\f{2\gamma}{\pi}+ \widetilde O(z^2\log(z)).\label{Y0 def}
\end{align}
Further, for $|z|>1 $, we have the representation (see, {\em e.g.}, \cite{AS})
\begin{align}\label{JYasymp2}
	 H_0^{+}(z)= e^{i z} \omega(z),\,\,\,\,\quad \omega(z)
	 =\widetilde O\big((1+|z|)^{-\frac{1}{2}}\big).
\end{align}

In the proofs of Theorem~\ref{thm:nonweighted} and Theorem~\ref{thm:weighted}, without loss of generality, we will perform all the analysis on $[\mu,\infty)$. Writing $z=\mu+\lambda^2$, $\lambda>0$, we have the limiting operators
\be\label{matrixfree}
\mathcal R_0^{\pm}(\mu+\lambda^2)(x,y)=\left[ \begin{array}{cc} R_0^\pm(\lambda^2)(x,y)   & 0 \\ 0 & - \frac{i}4 H_0^+(i\sqrt{2\mu+\lambda^2}|x-y|)\end{array}\right],
\ee
where
\be\label{scalarfree}
R_0^\pm(\lambda^2)(x,y)= \pm \frac{i}4  H_0^\pm(\lambda |x-y|)=\pm \frac{i}4 J_0(\lambda |x-y|)-\frac14 Y_0(\lambda |x-y|).
\ee
Thus, we have
\be\label{r0low2}
 \mathcal R_0^+(\mu+\lambda^2)(x,y)-\mathcal R_0^-(\mu+\lambda^2)(x,y) =  \frac{i}2 \left[ \begin{array}{cc} J_0(\lambda |x-y|) & 0 \\ 0 & 0 \end{array}\right].
\ee
We also have the bound,   with $R_2(\lambda^2)(x,y):=-\frac{i}{4}H_0^+(i\sqrt{2\mu+\lambda^2}|x-y|)$ and for $\lambda\geq 0$,
\be \label{R2 bounds}
| R_2(\lambda^2)(x,y)| \les 1+\log^{-}|x-y|,\,\,\,\text{ and } \,\, |\partial_\lambda^k R_2(\lambda^2)(x,y)| \les 1 ,\,\,\,\,\,k=1,2,...
\ee
To establish these bounds   consider  the cases $\sqrt{2\mu+\lambda^2}|x-y|<\frac{1}{2}$ and
$\sqrt{2\mu+\lambda^2}|x-y|>\frac{1}{2}$ separately.  For the first case we use
\eqref{J0 def} and \eqref{Y0 def} noting that $|x-y|<\mu^{-1/2}\les 1$, and that
$|\partial_\lambda^k \sqrt{2\mu+\lambda^2}|\les 1$.  For the latter case, using \eqref{JYasymp2}, the bound follows from the resulting exponential decay.

Below, using the properties of $R_0$ listed
above, we provide an expansion for the matrix free resolvent, $\mathcal R_0$, around $\lambda=0$ (i.e. $z=\mu$). In the next section, we will obtain analogous expansions for the perturbed resolvent.  Similar lemmas were proved in \cite{JN, Sc2, EG2} in the scalar case.
The following operators and the function arise naturally in the resolvent expansion 
(see \eqref{Y0 def})
\begin{align}
	G_0f(x)&=-\frac{1}{2\pi}\int_{\R^2} \log|x-y|f(y)\,dy, \label{G0 def}\\
\label{g form}
		g^{\pm}(\lambda)&:= \Big(\pm \frac{i}{4}-\frac{1}{2\pi}\log(\lambda/2)-\frac{\gamma}{2\pi} \Big)\\
\mathcal G_0(x,y)&=\left[\begin{array}{cc} G_0(x,y) &0\\ 0 & -\frac{i}4 H_0^+(i\sqrt{2\mu}|x-y|)
		\end{array}\right].\label{calG0}
\end{align}
Note that
\begin{align}
\mathcal G_0  =\left[\begin{array}{cc} -\Delta &0\\ 0 & \Delta-2\mu
		\end{array}\right]^{-1}=(\cH_0-\mu I)^{-1}.\label{G0delta}
\end{align}
Further, for notational convenience we
define the matrices
\begin{align*}
	M_{11}=\left[\begin{array}{cc} 1 &0\\ 0 &0
		\end{array}\right],\qquad
      M_{22}=\left[\begin{array}{cc} 0 &0\\ 0 &1
		\end{array}\right].
\end{align*}
We will use the notation $K(x,y)M_{11}$ or $KM_{11}$ to denote the operator with the convolution kernel
$$
\left[\begin{array}{cc} K(x,y) &0\\ 0 &0
		\end{array}\right],
$$
similar formula holds if $K$ is a matrix kernel.  We also use
the following notation, for a matrix operator $M$ if we write
$$
|M|\les f, \qquad \text{ or } \qquad M=O(f)
$$
with $f$ a scalar-valued function,
we mean that all entries of the matrix $M$   satisfy the
bound.
\begin{lemma}\label{R0 exp cor}
	We have the following expansion for the kernel of the free resolvent
	$$
		\mathcal R_0^{\pm}(\mu+\lambda^2)(x,y)=
		g^{\pm}(\lambda)M_{11}
		+\mathcal G_0(x,y)
		+E_0^{\pm}(\lambda)(x,y).
	$$
	Here $\mathcal G_0(x,y)$ is the kernel of the operator in \eqref{calG0}, $g^{\pm}(\lambda)$ is as in \eqref{g form},
	and the component functions of $E_0^{\pm}$ satisfy the bounds
	\begin{align*}
		|E_0^{\pm}|\les \la \lambda\ra^{\frac{1}{2} } \lambda^{\frac{1}{2} }\la x-y\ra^{\frac{1}{2} }, \qquad
		|\partial_\lambda E_0^{\pm}|\les \la \lambda\ra^{\frac{1}{2} } \lambda^{-\frac{1}{2} }\la x-y\ra^{\frac{1}{2} }, \qquad
		|\partial_\lambda^2 E_0^{\pm}|\les \la \lambda\ra^{\frac{1}{2} } \lambda^{-\frac{1}{2} }\la x-y\ra^{\frac{3}{2}}.
	\end{align*}
\end{lemma}

\begin{proof}
	The expansion of the scalar free resolvent was derived in   \cite[Lemma~3.1]{EG2}.  For the
	free resolvent evaluated at the imaginary argument, the proof easily follows from the properties of the Hankel
	function listed above.
\end{proof}

\begin{corollary}\label{R0 exp cor2}

	For $0<\alpha<1$ and $0<z_1<z_2<\lambda_1$ we have
	\begin{align*}
		|\partial_\lambda E_0^{\pm}(z_2)-\partial_\lambda E_0^{\pm}(z_1)|
		\les z_1^{-\frac{1}{2}}|z_2-z_1|^\alpha \la x-y\ra^{\frac{1}{2}+\alpha}
	\end{align*}

\end{corollary}

\section{Resolvent Expansion Around the Threshold $\mu$}  \label{sec:mu_exp}

It is convenient to write the potential matrix as $V=-\sigma_3 v v:=v_1v_2$ where
$v_1=-\sigma_3v$, $v_2=v$, and
\begin{align*}
	v=\frac12
\left[\begin{array}{cc}  \sqrt{V_1+V_2}+\sqrt{V_1-V_2}  &  \sqrt{V_1+V_2}-\sqrt{V_1-V_2} \\  \sqrt{V_1+V_2}-\sqrt{V_1-V_2}  &  \sqrt{V_1+V_2}+\sqrt{V_1-V_2} \end{array}\right]
=:\left[\begin{array}{cc} a & b\\ b & a\end{array}\right].
\end{align*}
By assumption A3), we have
\be\label{abdecay}
|a(x)|, |b(x)|\les \la x\ra^{-\beta/2},\,\,\,\text{ for some }\beta>3.
\ee

We employ the symmetric resolvent identity
\begin{align}\label{symm resolv id}
	\mathcal R_V^{\pm}(\mu+\lambda^2)=\mathcal R_0^{\pm}(\mu+\lambda^2)-\mathcal R_0^{\pm}(\mu+\lambda^2)
	v_1(M^{\pm}(\lambda))^{-1}
	v_2\mathcal R_0^\pm(\mu+\lambda^2),
\end{align}
where
\begin{align}\label{symm resolv id2}
	M^{\pm}(\lambda)=I+v_2\mathcal R_0^{\pm}(\mu+\lambda^2)v_1.
\end{align}

The key issue in the resolvent expansion around the threshold $\mu$  is the invertibility of the operator $M^\pm(\lambda)$ for small $\lambda$.
Using Lemma~\ref{R0 exp cor} in \eqref{symm resolv id},  we can write  $M^{\pm}(\lambda)$ as  
\be
	M^{\pm}(\lambda) = g^{\pm}(\lambda ) v_2 M_{11}v_1  +T 
		+v_2   E_0^{\pm}(\lambda)  v_1 ,\label{M def}
\ee
where $T$ is the transfer operator on $L^2\times L^2$ with the kernel
\begin{align}\label{T def}
	T(x,y)=I+v_2(x)\mathcal G_0(x,y) v_1(y).
\end{align}
Consider the contribution of the term with $g^{\pm}(\lambda)$ in \eqref{M def}.  Recalling the formulas for $v_1$ and $v_2$, we obtain
\begin{align*}
	 g^{\pm}(\lambda)&v_2 M_{11}v_1 =-g^{\pm}(\lambda)
 \left[\begin{array}{cc} a  & 0\\ b  & 0\end{array}\right]
\left[\begin{array}{cc} a  & b \\ 0 & 0\end{array}\right]=-\|a^2+b^2\|_{L^1(\R^2)}g^{\pm}(\lambda) P=:\widetilde{g}^{\pm}(\lambda) P,
\end{align*}
where $\widetilde{g}^{\pm}(\lambda):=-\|a^2+b^2\|_{L^1(\R^2)}g^{\pm}(\lambda)$, and $P$ is the  orthogonal projection onto the span  of the vector $(a,b)^T$ in $L^2\times L^2$.  More explicitly
\begin{align}\label{P def}
	P\left[\begin{array}{c} f\\ g\end{array}\right]
	&=\frac{1}{\|a^2+b^2\|_{L^1(\R^2)}}\left[\begin{array}{c} a \\ b \end{array}\right]
	\int_{\R^2} \big(a(y)f(y)+b(y)g(y)\big)\, dy.
\end{align}
This gives us the following expansion:

\begin{lemma}\label{M exp lemma}

	Let $0<\alpha<1$.
	For $\lambda>0$ with $M^{\pm}(\lambda)$, $P$ and $T$ as above.  Then
	\begin{align*}
		M^{\pm}(\lambda)=\widetilde g^{\pm}(\lambda)P+T+E_1^{\pm}(\lambda).
	\end{align*}
	Further, the error term, $E_1^\pm=v_2E_0^\pm v_1$, satisfies the bound
	\begin{multline*}
		\big\| \sup_{0<\lambda<\lambda_1} \lambda^{-\frac{1}{2}} |E_1^{\pm}(\lambda)|\big\|_{HS}
		+\big\| \sup_{0<\lambda<\lambda_1} \lambda^{\frac{1}{2}} |\partial_\lambda E_1^{\pm}(\lambda)|\big\|_{HS}	
		\\+\big\| \sup_{0<z_1<z_2<\lambda_1} z_1^{\frac{1}{2}} (z_2-z_1)^{-\alpha}
		|\partial_\lambda E_1^{\pm}(z_2)-\partial_\lambda E_1^\pm(z_1)|\big\|_{HS}	 \les 1,
	\end{multline*}
	provided that $a(x),b(x)\lesssim \langle x\rangle^{-\frac{3}{2}-\alpha-}$. Here $\|\cdot\|_{HS}$ is the Hilbert Schmidt   operator norm on
$L^2\times L^2$. 

\end{lemma}

\begin{proof}
The expansion is proven above.  The  bounds for $E_1^\pm=v_2E_0^\pm v_1$ follow from the bounds for $E_0^\pm$ in Lemma~\ref{R0 exp cor} and in Corollary~\ref{R0 exp cor2} since 
$$
\big\|\la x-y\ra^{\frac12+\alpha}\la x\ra^{-\frac32-\alpha-}\la y\ra^{-\frac32-\alpha-}\big\|_{L^2_xL^2_y}<\infty.
$$
\end{proof}

We make the following definitions.
\begin{defin}

	We say the operator $T:L^2\times L^2\to L^2\times L^2$ with kernel $T(\cdot,\cdot)$ is
	absolutely bounded if the operator with kernel $|T(\cdot,\cdot)|$ is bounded from
	$L^2\times L^2\to L^2\times L^2$.

\end{defin}
 
Note that Hilbert-Schmidt operators and finite rank operators are absolutely bounded.
  
\begin{defin} \label{def:regular}

	Let $Q=I-P$ be the projection orthogonal to the span of $(a,b)^T$. We say $\mu$ is a regular point of the spectrum $\mathcal H$ provided that
	$QTQ$ is invertible on $Q(L^2\times L^2)$. We denote $(QTQ)^{-1}$ by $QD_0Q$.

\end{defin}

Note that by the resolvent identity
$$
QD_0Q=Q-QD_0Qv_2\mathcal G_0v_1Q.
$$
Since $Q$ is a projection, it is absolutely bounded.  
By assumption~A3), \eqref{calG0}, \eqref{G0 def}, and \eqref{JYasymp2}, we have $|v_2\mathcal G_0v_1(x,y)|\les (1+|\log|x-y||) \la x\ra^{-3/2-} \la y\ra^{-3/2-}$. This implies that  $v_2\mathcal G_0v_1$ is a Hilbert-Schmidt operator. Therefore,  $QD_0Q$ is a sum of an absolutely bounded operator and an Hilbert-Schmidt operator, which is absolutely bounded. 

We also note the following orthogonality property of $Q$:
\be\label{Qorthogonality}
Qv_2M_{11}=M_{11}v_1Q=0.
\ee

In the scalar case, see e.g. \cite{JN,EG}, the invertibility
of $QTQ$ is related to the absence of distributional
$L^\infty$ solutions of $H\psi=0$. It is possible to prove a similar relationship for the matrix case. Define $S_1$ to be the Riesz projection onto the kernel of $QTQ$ as an operator
on $Q(L^2\times L^2)$.

\begin{lemma}\label{spectral lemma}

	If $|a(x)|+|b(x)|\les \la x\ra^{-1-}$ and if $\phi=(\phi_1,\phi_2)\in S_1(L^2\times L^2)$, then
	$\phi=v_1\psi$ where $\psi\in L^\infty\times L^\infty$ and
	$(\mathcal H-\mu I)\psi=0$ in the sense of distributions.

\end{lemma}
\begin{proof}

	Since $\phi\in S_1(L^2\times L^2)$, we have
	$Q\phi=\phi$. Also using $Q=I-P$, we obtain
	\begin{align*}
		0&=QTQ\phi=(I-P)T\phi= (I+v_2\mathcal G_0 v_1)\phi-P(I+v_2\mathcal G_0 v_1)\phi.
	\end{align*}
	Noting that $(a,b)^T=v_2(1,0)^T$, and that $P$ project onto the span of $(a,b)^T$, we have $PT\phi=c_0 v_2 (1,0)^T$ with $c_0$ a constant. Therefore,
	\begin{align*}
		\phi=-v_2\mathcal G_0v_1\phi +v_2(c_0,0)^T=v_2 \psi,
	\end{align*}
	where $\psi=- \mathcal G_0v_1\phi + (c_0,0)^T$.
	By assumption $|a(x)|+|b(x)|\les \la x\ra^{-1-}$ and
	$\phi\in L^2\times L^2$, and recalling \eqref{G0delta}, we have
	$$
	(\mathcal H_0-\mu I) \mathcal G_0(v_1\phi)= v_1\phi
	$$
	in the sense of distributions.
	It thus follows that
	\begin{align*}
	(\cH_0-\mu I) \psi =(\cH_0-\mu I)  [- \mathcal G_0v_1\phi + (c_0,0)^T] =-v_1\phi=-v_1v_2\psi= - V\psi.
	\end{align*}
Thus $(\cH-\mu I)\psi=0$.
	
	Now we prove that $\psi \in L^\infty \times L^\infty$. The first bound in  \eqref{R2 bounds} and the fact that the entries of $\phi$ are in $L^2$ and the entries of $v_2$ are in $L^\infty\cap L^2$ imply that the second entry of $\psi$ is bounded. We note that the first entry of $\psi$ is
	$$
	-\frac{1}{2\pi} \int_{\R^2} \log|x-y| (a(y),b(y))\phi(y) dy.
	$$
	Since $P\phi=0$, we can rewrite this as
	$$
	-\frac{1}{2\pi} \int_{\R^2} (\log|x-y|-\log|x|) (a(y),b(y))\phi(y) dy.
	$$
	The boundedness of this integral follows  immediately from the bound
	$$
	|\log|x-y|-\log|x||=\Big|\log\Big(\frac{|x-y|}{|x|}\Big)\Big|\les 1+\log\la y \ra +\log^-|x-y|.
 	$$
We refer the reader to 	 Lemma~5.1 of \cite{EG} for more details.

\end{proof}
It is also possible to prove a converse statement relating certain $L^\infty\times L^\infty
$ solutions of $(\cH-\mu I) \psi =0 $ to the non-invertibility of $QTQ$ as in Lemma~5.1 and Lemma~5.2 of \cite{EG} (also see \cite{JN}). 
We don't include these statements and proofs since they can be obtained from the scalar case as above.

The regularity assumption~A5) allows us to invert the operators $M^\pm(\lambda)$ for small $\lambda$ as follows:

\begin{lemma}\label{lem:Minverse}

	  Let $0<\alpha<1$.
    	Suppose that $\mu$ is a regular point of the spectrum of  $\mathcal H$. Then for   sufficiently small
	$\lambda_1>0$, the operators
	$M^{\pm}(\lambda)$ are invertible for all $0<\lambda<\lambda_1$ as bounded operators on $L^2\times L^2$.
	Further, one has
	\begin{align}
	\label{M size}
        	 M^{\pm}(\lambda)^{-1}=h^{\pm}(\lambda)^{-1}S+QD_0Q+ E^{\pm}(\lambda),
	\end{align}
	Here
   	$h^\pm(\lambda)=\widetilde g^\pm(\lambda)+c=-\|a^2+ b^2\|_{L^1}g^\pm(\lambda)+c$  (with $c\in\R$), and
  	\be\label{S_defn}
  	 	 S= P -PTQD_0Q -QD_0QTP + QD_0QTPTQD_0Q
  	\ee
	is a finite-rank operator with real-valued kernel.  Further, the error term satisfies the bounds
	\begin{multline*}
		\big\| \sup_{0<\lambda<\lambda_1} \lambda^{-\frac{1}{2} } |E^{\pm}(\lambda)|\big\|_{HS}
		+\big\| \sup_{0<\lambda<\lambda_1} \lambda^{\frac{1}{2} } |\partial_\lambda E^{\pm}(\lambda)|\big\|_{HS}	\\
		+\big\| \sup_{0<\lambda<\eta \les \lambda<\lambda_1}  \lambda^{\frac{1}{2}+\alpha} (\eta-\lambda)^{-\alpha}
		 |\partial_\lambda E^{\pm}(\eta)-\partial_\lambda E^\pm(\lambda)| \big\|_{HS}	
		\les 1,
	\end{multline*}
	provided that $a(x), b(x)\lesssim \langle x\rangle^{-\frac{3}{2}-\alpha-}$.

\end{lemma}

\begin{proof}

	We give a proof for the operator $M^+(\lambda)$, the expansion for $M^-(\lambda)$ is similar.  We drop the subscript
	`+' from the formulas.  Using Lemma~\ref{M exp lemma} with respect to the decomposition of
	$L^2\times L^2=P(L^2\times L^2)\oplus Q(L^2\times L^2)$,
	\begin{align*}
		M(\lambda)=\left[\begin{array}{cc} \tilde g(\lambda)P+PTP & PTQ \\ QTP & QTQ
		\end{array}\right]+E_1(\lambda).
	\end{align*}
	Denote the matrix component of the above equation by $A(\lambda)=\{a_{ij}(\lambda)\}_{i,j=1}^{2}$.

    Since  $QTQ$ is invertible by assumption, by the Fehsbach formula invertibility of
    $A(\lambda)$ hinges upon the existence
    of $d=(a_{11}-a_{12}a_{22}^{-1}a_{21})^{-1}$.
    Denoting $D_0=(QTQ)^{-1}:Q(L^2\times L^2)\to Q(L^2\times L^2)$, we have
	\begin{align*}
		a_{11}-a_{12}a_{22}^{-1}a_{21}= \widetilde g(\lambda)P+PTP-PTQD_0QTP  =h(\lambda) P
	\end{align*}
	with $h(\lambda)=\widetilde g(\lambda)+Tr(PTP-PTQD_0QTP)=\widetilde g(\lambda)+c$,
	where $c\in\R$ as the kernels of $T$, $QD_0Q$ and
	$v_1, v_2$ are real-valued. The invertibility of this operator on $PL^2$ for small $\lambda$ follows from
	\eqref{g form}.
	Thus, by the
	Fehsbach formula,
	\begin{align}\nonumber
		A(\lambda)^{-1}&=\left[\begin{array}{cc} d & -da_{12}a_{22}^{-1}\\
		-a_{22}^{-1}a_{21}d & a_{22}^{-1}a_{21}da_{12}a_{22}^{-1}+a_{22}^{-1}
		\end{array}\right]\\
		&=h^{-1}(\lambda)\left[\begin{array}{cc} P & -PTQD_0Q\\ -QD_0QTP & QD_0QTPTQD_0Q
		\end{array}\right]+QD_0Q =: h^{-1}(\lambda)S+QD_0Q. \label{Ainverse}
	\end{align}
    Note that $S$ has finite rank. This and the absolute boundedness of $QD_0Q$ imply that $A^{-1}$
    is absolutely bounded.  To avoid confusion, we will write $S$ as a sum of four components rather than in
    a matrix form.

    	Finally, we write
    	$$
    		M(\lambda)=A(\lambda)+E_1(\lambda)=[\mathbbm{1}+E_1(\lambda) A^{-1}(\lambda)] A(\lambda).
    	$$
	Therefore, by a Neumann series expansion, we have
	\be\label{M plus S}
        M^{-1}(\lambda) =A^{-1}(\lambda)
        \big[\mathbbm{1}+E_1(\lambda) A^{-1}(\lambda)\big]^{-1}=h(\lambda)^{-1}S
        +QD_0Q+E(\lambda),
	\ee
  	The error bounds follow in light of the bounds for $E_1(\lambda)$ in Lemma~\ref{M exp lemma}
	and the fact that, as an absolutely bound operator on $L^2$, $|A^{-1}(\lambda)|\les 1$,
	$|\partial_\lambda  A^{-1}(\lambda)|\les \lambda^{-1}$, and (for $0<\lambda<\eta<\lambda_1$)
  	$$|\partial_\lambda  A^{-1}(\lambda)-\partial_\lambda
	A^{-1}(\eta)|\les (\eta-\lambda)^\alpha \lambda^{-1-\alpha}.$$
  	In the Lipschitz estimate, the factor $\lambda^{-\f12-\alpha}$  arises from the case when
	the derivative hits $A^{-1}(\lambda)$.

\end{proof}

We finish this section by noting that, using Lemma~\ref{lem:Minverse} in \eqref{symm resolv id}, one gets
\be\label{resolve}
\mathcal R_V^{\pm}(\mu+\lambda^2)=\mathcal R_0^{\pm}(\mu+\lambda^2)-\mathcal R_0^{\pm}(\mu+\lambda^2)
	v_1[h^\pm(\lambda)^{-1}S+QD_0Q+E^{\pm}(\lambda)]
	v_2\mathcal R_0^\pm(\mu+\lambda^2).
\ee

\section{Proof of Theorem~\ref{thm:weighted} for energies close to $\mu$}\label{sec:lowweighted}

Let $\chi$ be a smooth cut-off for $[0,\lambda_1]$, where $\lambda_1$ is sufficiently small so that the expansions in the previous section are valid.
We have
\begin{theorem}\label{lowprop} Fix $0<\alpha<1/4$. Let $|a(x)|+|b(x)|\lesssim \la x\ra^{-\f32-\alpha-}.$
For any $t > 2$, we have
\be\label{stone2}
\Big|\int_0^\infty e^{it\lambda^2}\lambda \chi(\lambda)  [\mathcal R_V^+(\mu+\lambda^2)-\mathcal R_V^-(\mu+\lambda^2)](x,y)\,
d\lambda\Big| \les \frac{\sqrt{w(x)w(y)}}{t\log^2(t)}+\f{\la x\ra^{\f32 } \la y\ra^{\f32 }}{t^{1+\alpha}}.
\ee
\end{theorem}

In the proof of this theorem we need the following Lemmas, which are standard and their proofs can be found in  \cite{EG2}.
\begin{lemma} \label{lem:ibp} For $t>2$, we have
$$
\Big|\int_0^\infty e^{it\lambda^2} \lambda \, \mathcal E(\lambda) d\lambda -\f{i\mathcal E(0)}{2t}\Big| \les \f1t\int_0^{t^{-1/2}}|\mathcal E^\prime(\lambda)| d\lambda+ \Big|\frac{\mathcal{E}^\prime(t^{-1/2})}{t^{3/2}}\Big|
+\f1{t^2}\int_{t^{-1/2}}^\infty \Big|\Big(\frac{\mathcal E^\prime(\lambda)}{\lambda}\Big)^\prime\Big| d\lambda.
$$
\end{lemma}

\begin{lemma}\label{lem:ibp2} Assume that  $\mathcal E(0)=0$. For $t>2$, we have
\be\label{ibp2}
\Big|\int_0^\infty e^{it\lambda^2} \lambda \, \mathcal E(\lambda) d\lambda  \Big|
\les \f1t\int_0^{\infty}\frac{|\mathcal E^\prime(\lambda)|}{  (1+\lambda^2 t)} d\lambda
+\f1{t}\int_{t^{-1/2}}^\infty \big|  \mathcal E^\prime(\lambda \sqrt{1+ \pi t^{-1}\lambda^{-2}} )-\mathcal E^\prime(\lambda)  \big| d\lambda.
\ee
\end{lemma}

We start with the contribution of the free resolvent in \eqref{resolve} to \eqref{stone2}. Recall   \eqref{r0low2}:
$$
\mathcal R_0^+(\mu+ \lambda^2)(x,y)-\mathcal R_0^-(\mu+\lambda^2)(x,y)=\frac{i}2 J_0(\lambda|x-y|)M_{11}.
$$
Therefore, the following proposition follows from the corresponding bound for the scalar free resolvent, Proposition~4.3 in \cite{EG2}. The proof uses Lemma~\ref{lem:ibp} with $\mathcal E(\lambda)=\frac{i}2 J_0(\lambda|x-y|)$.
\begin{prop}\label{freeevol} We have
$$\int_0^\infty e^{it\lambda^2}\lambda \chi(\lambda)  [\mathcal R_0^+(\mu+\lambda^2)-\mathcal R_0^-(\mu+\lambda^2)](x,y)
\, d\lambda=-\frac1{4t} M_{11} +O\Big(\f{\la x \ra^{\f32 } \la y\ra^{\f32 }}{t^{\f54}}\Big).
$$
\end{prop}

Now consider the contribution of the term involving $(h^{\pm})^{-1}S$   in \eqref{resolve} to \eqref{stone2}. Using Lemma~\ref{R0 exp cor} we have
\begin{multline}\label{Scontribution}
	 \frac{\mathcal R_0^{+}  v_1Sv_2\mathcal R_0^{+} }{h^+ }
	 -\frac{\mathcal R_0^{-} v_1Sv_2\mathcal R_0^{-} }{h^- } 
	 =  \big(\frac{(g^+)^2}{h^+ }
	 - \frac{(g^-)^2}{h^- }  \big) M_{11}v_1Sv_2M_{11} \\
	+ \big(\frac{1}{h^+ } - \frac{1}{h^- }  \big) \mathcal G_0 v_1Sv_2\mathcal G_0 
	+\big(\frac{g^+  }{h^+ } - \frac{g^-  }{h^- }  \big)
	(M_{11}v_1Sv_2\mathcal G_0+\mathcal G_0v_1Sv_2M_{11})
	+E_2^+ -E_2^-,
\end{multline}
where
\begin{align*}
	E_2^{\pm} =\frac{E_0^\pm v_1Sv_2 \big(g^\pm M_{11} +\mathcal G_0\big) }{h^\pm } +
	\frac{\big(g^\pm M_{11}+\mathcal G_0\big) v_1Sv_2 E_0^{\pm}   }{h^\pm }
	+\frac{E_0^{\pm} v_1Sv_2 E_0^{\pm}  }{h^{\pm}  }.
\end{align*}
Using the orthogonality property \eqref{Qorthogonality} and the definition \eqref{S_defn} of $S$, we obtain
\begin{align*}
	M_{11} v_1 S v_2 M_{11}=M_{11} v_1 P v_2 M_{11}=-\|a^2+b^2\|_{L^1(\R^2)} M_{11}.
\end{align*}
Also recall that $h^\pm(\lambda)=-\|a^2+b^2\|_{L^1} g^\pm(\lambda) +c$, $c\in\R$, and (from \eqref{g form})  that $g^{+}(\lambda)=-\frac1{2\pi}\log\lambda+z$ with
$g^-(\lambda)=\overline{g^+(\lambda)}$ and $z-\overline{z}=\frac{i}{2}$. Therefore we can  write
\begin{multline}\label{new28}
	\eqref{Scontribution}=\frac{i}{2}M_{11}+\frac{ia_1}{(\log(\lambda)+b_1)^2+c_1^2} M_{11}
	+\frac{ia_2}{(\log(\lambda)+b_2)^2+c_2^2}\mathcal G_0 v_1Sv_2\mathcal G_0\\
	+\frac{ia_3}{(\log(\lambda)+b_3)^2+c_3^2}
	(M_{11}v_1Sv_2\mathcal G_0+\mathcal G_0v_1Sv_2M_{11})+E_2^+(\lambda)-E_2^-(\lambda),
\end{multline}
where $a_i,b_i,c_i$ are real. Using this the following proposition will follow from the bounds obtained in \cite{EG2}.
\begin{prop}\label{Scontr} Let $0<\alpha<1/4$. If $|a(x)|+|b(x)|\les \la x\ra^{-\f32-\alpha-}$, then
we have
\begin{multline*}\int_0^\infty e^{it\lambda^2}\lambda \chi(\lambda)  \Big[\frac{\mathcal R_0^{+}(\mu+\lambda^2)v_1Sv_2\mathcal R_0^{+}(\mu+\lambda^2)}{h^+(\lambda)}
	 -\frac{\mathcal R_0^{-}(\mu+\lambda^2)v_1Sv_2\mathcal R_0^{-}(\mu+\lambda^2)}{h^-(\lambda)}\Big](x,y)
	\, d\lambda\\
=-\frac1{4t} M_{11} +  O\Big( \f{\sqrt{w(x)w(y)}  }{t\log^2(t)} \Big)+O\Big( \f{\la x\ra^{\f12+\alpha+} \la y\ra^{\f12+\alpha+}}{t^{1+\alpha} } \Big).
\end{multline*}
\end{prop}
\begin{proof}
First consider the contribution of the first term in \eqref{new28}:
\begin{align*}
	\frac{i}{2}M_{11}\int_0^\infty e^{it\lambda^2} \lambda \chi(\lambda)\, d\lambda &=-\frac{1}{4t}M_{11}+O(t^{-2}),
\end{align*}
where the equality follows from Lemma~\ref{lem:ibp}.

The contribution of the second  summand  in \eqref{new28} can be handled using the bound
\begin{align}\label{log2t}
	\int_0^\infty e^{it\lambda^2} \lambda \frac{\chi(\lambda)}{(\log(\lambda)+c_1)^2+c_2^2}\, d\lambda
	 =O(t^{-1}(\log t)^{-2}), \qquad t>2,
\end{align}
which is essentially Lemma~4.5 in \cite{EG2} and it is proved by using Lemma~\ref{lem:ibp}.

The contribution of the third (similarly the fourth)  summand   in \eqref{new28} can also be handled using \eqref{log2t} along with the bound 
\begin{multline}\label{g0g0}
 \big|\mathcal G_0 v_1Sv_2\mathcal G_0 (x,y)\big| \\ \leq \big\||S|\big\|_{L^2\to L^2} \big\|  \mathcal G_0(x,x_1) v_1(x_1)  \big\|_{L^2_{x_1}} 
 \big\|  \mathcal G_0(y,y_1) v_2(y_1)  \big\|_{L^2_{y_1}} \les \sqrt{w(x)w(y)}.
\end{multline}
The last inequality follows from the absolute boundedness of $S$, the bound  
\begin{align}\label{G0 bound}
	|\mathcal G_0(x,x_1)|\les 1+|\log|x-x_1||\les \sqrt{w(x)}+ k(x,x_1), 
\end{align}
where $k(x,x_1)=1+\log^-|x-x_1|+\log^+|x_1|$, 
and
$$
 \big\|\big(\sqrt{w(x)}+ k(x,x_1)\big) \la x_1\ra^{-3/2} \big\|_{L^2_{x_1}}\les \sqrt{w(x)}.
$$

We now consider the error term, $E_2^{\pm}(\lambda)$. 
Note that 
$$
\Big|\frac{g^\pm(\lambda)}{h^\pm(\lambda)}\Big|=\Big|c_1-\frac{c_2}{h^\pm(\lambda)}\Big|\les 1,\,\,\,\,\,\Big|\partial_\lambda^k \frac{g^\pm(\lambda)}{h^\pm(\lambda)}\Big|\les \frac{1}{\lambda^k}, \,\,\,k=1,2,3,...
$$
Using this, the absolute boundedness of $S$, the decay bounds $|a(x)|+|b(x)|\les \la x\ra^{-\frac32-\alpha-}$,  the bound \eqref{G0 bound},
and the bounds in Lemma~\ref{R0 exp cor} and Corollary~\ref{R0 exp cor2} as in the proof of \eqref{g0g0}, we obtain (for $0<\lambda<\eta\les \lambda<\lambda_1$)
\begin{align*}
	&|\partial_\lambda (\chi(\lambda)E_2^{\pm}(\lambda))(x,y)| \les  \chi(\lambda)
	 (\la x\ra   \la y\ra)^{\frac12}\lambda^{-\frac{1}{2}},\\
	&\big|\big(\partial_\lambda (\chi(\eta)E_2^{\pm}(\eta))-\partial_\lambda (\chi(\lambda)E_2^{\pm}(\lambda))\big)(x,y)\big|
	\les \chi(\lambda)   (\la x\ra   \la y\ra)^{\frac{1}{2}+\alpha}
	\lambda^{-\frac{1}{2}-\alpha}(\eta-\lambda)^\alpha.
\end{align*}
Therefore the contribution of the error term is controlled by using  Lemma~\ref{lem:ibp2} as in Lemma~4.6 of \cite{EG2}.
\end{proof}

Now we consider the contribution of the term $QD_0Q$  in \eqref{resolve} to \eqref{stone2}.
\begin{prop}\label{QDQcontr} Let $0<\alpha<1/4$. If $|a(x)|+|b(x)|\les \la x\ra^{-\f32-\alpha-}$, then
we have
$$\int_0^\infty e^{it\lambda^2}\lambda \chi(\lambda)  \big[ \mathcal R_0^{+} v_1QD_0Qv_2\mathcal R_0^{+} 
	 - \mathcal R_0^{-} v_1QD_0Qv_2\mathcal R_0^{-}  \big](x,y)\,
d\lambda 
= O\Big( \f{\la x\ra^{\f12+\alpha+} \la y\ra^{\f12+\alpha+}}{t^{1+\alpha} } \Big).
$$\end{prop}
\begin{proof}
Using Lemma~\ref{R0 exp cor} and \eqref{Qorthogonality} we have
\begin{multline}\label{QDQcontribution}
	  \mathcal R_0^{+} v_1QD_0Qv_2\mathcal R_0^{+} 
	 -\mathcal R_0^{-} v_1QD_0Qv_2\mathcal R_0^{-}  
	 =  \mathcal G_0 v_1QD_0Qv_2 (E_0^+ -E_0^- ) \\ +(E_0^+ -E_0^- )  v_1QD_0Qv_2 \mathcal G_0
+E_0^+  v_1QD_0Qv_2E_0^+ -E_0^-  v_1QD_0Qv_2E_0^- =:E_3.
\end{multline}
Since $QD_0Q$ is absolutely bounded, $E_3$ satisfies the same bounds that we obtained for the error term $E_2$ above.  
\end{proof}

Finally the contribution of $E^{\pm}(\lambda)$ in \eqref{resolve} to \eqref{stone2} can be handled exactly as in Proposition~4.9 of \cite{EG2}:
\begin{prop}\label{Econtr} Let $0<\alpha<1/4$. If $|a(x)|+|b(x)|\les \la x\ra^{-\f32-\alpha-}$, then
we have
\begin{multline*}\int_0^\infty e^{it\lambda^2}\lambda \chi(\lambda)  \big[ \mathcal R_0^{+}(\lambda^2)v_1E^+(\lambda)v_2\mathcal R_0^{+}(\lambda^2)
	 - \mathcal R_0^{-}(\lambda^2)v_1E^-(\lambda)v_2\mathcal R_0^{-}(\lambda^2) \big](x,y)\,
d\lambda\\
= O\Big( \f{\la x\ra^{\f12+\alpha+} \la y\ra^{\f12+\alpha+}}{t^{1+\alpha} } \Big).
\end{multline*}
\end{prop}

This finishes the proof of Theorem~\ref{lowprop}.

\section{Proof of Theorem~\ref{thm:weighted} for energies separated from the thresholds}

In this section we complete the proof of Theorem~\ref{thm:weighted} by proving
\begin{theorem} \label{thm:mainineq high}
	Under the assumptions of Theorem~\ref{thm:weighted}, we have
	for $t>2$
	\be\label{weighteddecayhi}
		\sup_{L\geq1}\bigg| \int_0^\infty e^{it\lambda^2}\lambda \widetilde{\chi}(\lambda) \chi(\lambda/L)
		 [\mathcal R_V^+(\mu+\lambda^2)-\mathcal R_V^-(\mu+\lambda^2)](x,y) d\lambda \bigg|
		 \les\f{\la x\ra^{\f32}\la y\ra^{\f32}}{t^{\f32}}
	\ee
where $\widetilde\chi=1-\chi$.
\end{theorem}

We employ the resolvent expansion
\be 
     \mathcal R_V^{\pm}  =\sum_{m=0}^{2M+2} \mathcal R_0^{\pm} 
    (-V\mathcal R_0^{\pm} )^m \label{born series} 
     +\mathcal R_0^{\pm} (V\mathcal R_0^{\pm} )^M V\mathcal R_V^{\pm} 
    V (\mathcal R_0^{\pm} V)^M \mathcal R_0^{\pm}.
\ee
We first note that the contribution of the term $m=0$ can be 
bounded by
$\frac{\la x\ra^{\f32}\la y\ra^{\f32}}{t^2}$ by integrating by parts twice (there are no boundary terms because of the cutoff).  We approach the energies separated from zero differently from the
small energies.  In particular, we won't use Lemma~\ref{R0 exp cor}, but instead employ a component-wise
approach.  Recall that
\begin{align*}
	\mathcal R_0^{\pm}(\mu+\lambda^2)(x,y)=\left[\begin{array}{cc} R_0^{\pm}(\lambda^2)(x,y) & 0\\
	0 & -\frac{i}{4}H_0^+(i\sqrt{2\mu+\lambda^2}|x-y|)\end{array}\right]
\end{align*}

For the case $m>0$ we won't make use of any  cancelation between `$\pm$' terms. Thus, we will only consider $R_0^-$, and drop the `$\pm $' signs.
Using \eqref{R0 def}, \eqref{J0 def}, \eqref{Y0 def}, and \eqref{JYasymp2} we write
\begin{align}\label{freeform}
    R_0(\lambda^2)(x,y)=e^{-i\lambda|x-y|}\rho_+(\lambda|x-y|)+\rho_-(\lambda|x-y|),
\end{align}
where $\rho_+$ and $\rho_-$ are supported on the sets $[1/4,\infty)$ and $[0, 1/2]$, respectively.
Moreover, we have the bounds
\be\label{omega bds}
    \rho_-(y) =\widetilde O(1+|\log y|),\,\,\,\,\,\,\, \rho_+(y) =\widetilde O\big( (1+|y|)^{-1/2}\big)\\
\ee
This controls the top left component of the matrix operator.  The lower right term can be similarly controlled as
\begin{align*}
    H_0^+(i\sqrt{2\mu+\lambda^2})(x,y)=e^{-\sqrt{2\mu+\lambda^2}|x-y|}\rho_+(\sqrt{2\mu+\lambda^2}|x-y|)+\rho_-(\sqrt{2\mu+\lambda^2}|x-y|).
\end{align*}
As such we can write
\begin{multline*}
	\mathcal R_0^{\pm}(\mu+\lambda^2)(x,y)=e^{-i\lambda|x-y|} \left[\begin{array}{cc} \rho_+(\lambda |x-y|) & 0\\
	0 & e^{(i\lambda-\sqrt{2\mu+\lambda^2})|x-y|}\rho_+(\sqrt{2\mu+\lambda^2}|x-y|)\end{array}\right]\\ +
	 \left[\begin{array}{cc} \rho_-(\lambda |x-y|) & 0\\
	0 & \rho_-(\sqrt{2\mu+\lambda^2}|x-y|)\end{array}\right]
\end{multline*}
It is easy to see that
\begin{align*}
e^{(i\lambda-\sqrt{2\mu+\lambda^2})|x-y|}\rho_+(\sqrt{2\mu+\lambda^2}|x-y|)&=\widetilde O\big(\rho_+(\lambda|x-y|)\big),   \\
\rho_-(\sqrt{2\mu+\lambda^2}|x-y|)&=\widetilde O\big(\rho_-(\lambda|x-y|)\big).
\end{align*}
Therefore, we can use the right hand side of \eqref{freeform} for each component of $\mathcal R_0$.
The argument for the high energy now proceeds as in Section~5
of \cite{EG2}.  We provide a sketch of the details for the convenience of the reader.


We first control the contribution of the finite born series in \eqref{born series} for $m>0$.
Note that the contribution of the $m$th term of \eqref{born series} to the integral in \eqref{weighteddecayhi} can be written as a finite sum of integrals of the form
\be
		\int_{\R^{2m}} \int_0^\infty e^{it\lambda^2}\lambda \mathcal E(\lambda)
		\prod_{n=1}^{m} W(x_n) \, d\lambda \, dx_1\dots dx_{m},\label{high born1}
\ee
where $d_j=|x_{j-1}-x_j|$, $J\cup J^*$ is a partition of $\{1,...,m,m+1\}$, and
$$
\mathcal E(\lambda) := \widetilde \chi(\lambda)
		\chi(\lambda/L) e^{- i \lambda \sum_{j\in J}d_j} \prod_{j\in J} \rho_+
		(\lambda d_j)
		\prod_{\ell\in J^*} \rho_-(\lambda d_\ell).
$$
Here, with a slight abuse of notation,  $W(x)$ denotes either $\pm V_1(x)$ or $\pm V_2(x)$ (since we only use the decay assumption and do not rely on cancelations, this shouldn't create any confusion).

To estimate the derivatives of $\mathcal E$, we note that
\begin{align*}
\big|\partial_\lambda^k \big[\rho_+(\lambda d_j)\big]\big|&\les \frac{d_j^k}{(1+\lambda d_j)^{k+1/2}},\,\,\,\,\,\,\,k=0,1,2,...,\\
\big|\partial_\lambda^k \big[\rho_-(\lambda d_j)\big]\big|&\les \f1{\lambda^k},\,\,\,\,k=1,2,...
\end{align*}
Using the monotonicity of $\log^-$ function, we also obtain
$$
\widetilde\chi(\lambda) \big|\rho_-(\lambda d_j)\big|\les \widetilde \chi(\lambda) (1+|\log(\lambda d_j)|) \chi_{\{0<\lambda d_j\leq 1/2\}}
\les \widetilde \chi(\lambda) (1+\log^-(\lambda d_j))\les
1+\log^-(d_j).
$$
It is also easy to see that
$$
\Big|\frac{d^k}{d\lambda^k} \chi(\lambda/L)\Big|\les \lambda^{-k}.
$$
Finally, noting that
 $(\widetilde \chi)^{\prime}$ is supported on the set $\{\lambda\approx 1\}$, we can estimate
\begin{multline} \label{ehiprime}
\big|\partial_\lambda \mathcal E\big| \les \widetilde \chi(\lambda) \Big(\f1\lambda +\sum_{k\in J}\big(d_k + \frac{d_k}{1+\lambda d_k} \big)\Big)
\prod_{j\in J} \frac{1}{(1+\lambda d_j)^{1/2}}
		\prod_{\ell\in J^*} (1+\log^-(d_\ell))\\
\les  \widetilde \chi(\lambda) \Big(\f1\lambda+\sum_{k\in J} \frac{d_k}{(1+\lambda d_k)^{1/2}}  \Big)
		\prod_{\ell\in J^*} (1+\log^-(d_\ell))\\
\les  \widetilde \chi(\lambda) \Big(\lambda^{-1}+\sum_{k\in J} d_k^{\f12} \lambda^{-\f12}  \Big)
		\prod_{\ell\in J^*} (1+\log^-(d_\ell))
\les \widetilde \chi(\lambda) \lambda^{-\f12} \prod_{k=0}^{m+1} \la x_k\ra^{\f12}
		\prod_{\ell=1}^{m+1} (1+\log^-(d_\ell)).
\end{multline}
We also have
\begin{multline} \label{ehiprime2}
\big|\partial_\lambda^2 \mathcal E\big| \les \widetilde \chi(\lambda) \Big(\f1{\lambda^2}+\sum_{k\in J}\big(d_k^2 + \frac{d_k^2}{(1+\lambda d_k)^2}\big)   \Big)
\prod_{j\in J} \frac{1}{(1+\lambda d_j)^{1/2}}
		\prod_{\ell\in J^*} (1+\log^-(d_\ell))\\
\les  \widetilde \chi(\lambda) \Big(\lambda^{-2}+\sum_{k\in J} d_k^{\f32} \lambda^{-\f12}   \Big)
		\prod_{\ell\in J^*} (1+\log^-(d_\ell))\les \widetilde \chi(\lambda) \lambda^{-\f12} \prod_{k=0}^{m+1} \la x_k\ra^{\f32}
		\prod_{\ell=1}^{m+1} (1+\log^-(d_\ell)).
\end{multline}
Using Lemma~\ref{lem:ibp2} (and taking the support condition of $\widetilde \chi$ into account), we can bound the $\lambda$ integral in \eqref{high born1} by
\be\label{hilip}
\f1{t^2}\int_0^{\infty}\frac{|\mathcal E^\prime(\lambda)|}{ \lambda^2 } d\lambda
+\f1{t}\int_{0}^\infty \big|  \mathcal E^\prime(\lambda \sqrt{1+ \pi t^{-1}\lambda^{-2}} )-\mathcal E^\prime(\lambda)  \big| d\lambda,
\ee
Using \eqref{ehiprime}, we can bound the first integral in \eqref{hilip} by
\be\label{hilip1}
 \prod_{k=0}^{m+1} \la x_k \ra^{\f12} \prod_{\ell=1}^{m+1} (1+\log^-(d_\ell))\int_{0}^\infty \widetilde \chi(\lambda) \lambda^{-5/2} d\lambda\les \prod_{k=0}^{m+1} \la x_k \ra^{\f12} \prod_{\ell=1}^{m+1} (1+\log^-(d_\ell)).
\ee
To estimate the second integral in \eqref{hilip} first note that
\be\label{bminusa}
\lambda \sqrt{1+ \pi t^{-1}\lambda^{-2}}  - \lambda \approx \f1{t\lambda}.
\ee
Next  using \eqref{bminusa}, \eqref{ehiprime} and \eqref{ehiprime2}, we have (for any $0\leq \alpha \leq 1 $)
\begin{multline} \label{ehidiff}
 \big|  \mathcal E^\prime(\lambda \sqrt{1+ \pi t^{-1}\lambda^{-2}} )-\mathcal E^\prime(\lambda)  \big| \\ \les
 \widetilde\chi(2\lambda) \lambda^{-\f12} \prod_{k=0}^{m+1} \la x_k\ra^{\f12}
		\prod_{\ell=1}^{m+1} (1+\log^-(d_\ell)) \,
 \min\Big(1, \f1{t\lambda} \prod_{k=0}^{m+1} \la x_k\ra\Big)\\
\les t^{-\alpha} \widetilde\chi(2\lambda) \lambda^{-\f12-\alpha} \prod_{k=0}^{m+1} \la x_k\ra^{\f12+\alpha}
		\prod_{\ell=1}^{m+1} (1+\log^-(d_\ell)).
\end{multline}
Using this bound for $\alpha\in(1/2,1]$, we bound the second integral in \eqref{hilip} by
\begin{multline} \label{hilip2}
t^{-\alpha} \prod_{k=0}^{m+1} \la x_k\ra^{\f12+\alpha}
		\prod_{\ell=1}^{m+1} (1+\log^-(d_\ell)) \int_0^\infty \widetilde\chi(2\lambda) \lambda^{-\f12-\alpha} \les  \\ \les t^{-\alpha} \prod_{k=0}^{m+1} \la x_k\ra^{\f12+\alpha}
		\prod_{\ell=1}^{m+1} (1+\log^-(d_\ell)).
\end{multline}
Combining \eqref{hilip1} and \eqref{hilip2}, we obtain
$$
|\eqref{hilip}|\les  t^{-1-\alpha} \prod_{k=0}^{m+1} \la x_k\ra^{\f12+\alpha}
		\prod_{\ell=1}^{m+1} (1+\log^-(d_\ell))
$$
Using this (with $\f12<\alpha<2\beta-\f52$) in \eqref{high born1}, we obtain
\begin{multline*}
|\eqref{high born1}|\les  t^{-1-\alpha}
		\int_{\R^{2m}}  \prod_{k=0}^{m+1} \la x_k\ra^{\f12+\alpha}
		\prod_{\ell=1}^{m+1} (1+\log^-(d_\ell))
		\prod_{n=1}^{m} |V(x_n)|   \, dx_1\dots dx_{m}\\ \les \f{\la x_0\ra^{\f12+\alpha}\la x_{m+1}\ra^{\f12+\alpha}}{t^{\f32}}.
\end{multline*}

To control the contribution of the remainder term in \eqref{born series}, we will employ the limiting absorption principle, \eqref{free lap} and \eqref{lap}, both for $\mathcal R_0$ and $\mathcal R_V$.

Using the representation \eqref{omega bds}, and the discussion following it,
we note the following bounds hold   on $\lambda>\lambda_1>0$,
\begin{multline*}
	|\partial_\lambda^k \mathcal R_0^{\pm}(\mu+\lambda^2)(x,y)|\les \la x-y\ra^k \left\{\begin{array}{ll}
	|\log(\lambda|x-y|)| & 0<\lambda|x-y|<\frac{1}{2}\\
	(\lambda|x-y|)^{-\frac{1}{2}} & \lambda|x-y|\gtrsim 1\end{array}\right.\\
	\les \lambda^{-\frac{1}{2}}|x-y|^{-\frac{1}{2}} \la x-y\ra^k.
\end{multline*}
Thus, for $\sigma>\f12+k$,
\begin{align}\label{R0 wtdL2}
	\|\partial_\lambda^k \mathcal R_0^{\pm}(\mu+\lambda^2)(x,\cdot) \|_{X_{-\sigma}}
	\les\lambda^{-\f12} \Big[\int_{\R^2} \frac{\la x-y\ra^{2k}|x-y|^{-1}}{\la y\ra^{2\sigma}}\, dy\Big]^{\f12}
	\les \lambda^{-\f12} \la x\ra^{\max(0,k-1/2)}.
\end{align}
Once again, we estimate the  '$\pm$' terms separately and omit the `$\pm$' signs.

	We  write the contribution of  the remainder term in \eqref{born series}   to \eqref{weighteddecayhi} as
	\begin{align}\label{I def}
		 \int_0^\infty e^{it\lambda^2} \lambda \, \mathcal E(\lambda)(x,y)\, d\lambda,
	\end{align}
where
\begin{multline}\label{tailcont}
\mathcal E(\lambda)(x,y)    = \widetilde{\chi}(\lambda) \chi(\lambda/L) \times \\	\big\la V\mathcal R_V^{\pm}(\mu+\lambda^2) V (\mathcal R_0^{\pm}(\mu+\lambda^2)V)^M \mathcal R_0^{\pm}(\mu+\lambda^2)(\cdot,x), (\mathcal R_0^{\pm}(\mu+\lambda^2)V)^M \mathcal R_0^{\pm}(\mu+\lambda^2)(\cdot,y) \big\ra.
\end{multline}
Using \eqref{lap}, \eqref{free lap}, and \eqref{R0 wtdL2} (provided that $M\geq 2$) we see that
	\begin{align}\label{a deriv bds}
		\big|\partial_\lambda^k \mathcal E(\lambda)(x,y)\big|
		&\les \widetilde{\chi}(\lambda) \chi(\lambda/L) \la \lambda \ra^{-2-}\la x\ra^{\f32}\la y\ra^{\f32}, \qquad k=0,1,2.
	\end{align}
	This requires that $|V(x)|\les \la x\ra^{-3-}$. One can see that the requirement on the decay rate of the potential arises when, for instance,
	both $\lambda$ derivatives act on one resolvent, this twice differentiated resolvent operator maps
	$X_{\frac{5}{2}+}\to X_{-\frac{5}{2}-}$ by \eqref{lap},  or is in $X_{-\frac{5}{2}-}$ by \eqref{R0 wtdL2}.
	The potential then needs to map
	$X_{-\frac{5}{2}-} \to  X_{ \frac{1}{2}+}$  for the next application of the limiting absorption principle.
	This is satisfied if $|V(x)|\les \la x\ra^{-3-}$.
	
	The required bound now follows by integrating by parts twice:
	\begin{align*}
		|\eqref{I def}|\les |t|^{-2} \int_0^\infty \bigg|\partial_\lambda
		\bigg(\frac{\partial_\lambda \,\mathcal E(\lambda)(x,y)}{\lambda}\bigg)\bigg|\, d\lambda
		\les |t|^{-2} \la x\ra^{\frac{3}{2}}\la y\ra^{\frac{3}{2}}.
	\end{align*}

\section{Proof of Theorem~\ref{thm:nonweighted} for energies close to $\mu$}

In this section we will prove the following
\begin{theorem}\label{lowprop2} Under the conditions of Theorem~\ref{thm:nonweighted}, we have
\be\label{stone3}
\sup_{x,y\in\R^2}\Big|\int_0^\infty e^{it\lambda^2}\lambda \chi(\lambda)  [\mathcal R_V^+(\mu+\lambda^2)-\mathcal R_V^-(\mu+\lambda^2)](x,y)\,
d\lambda\Big| \les \frac{1}{t},\,\,\,\,\,\,t>0
\ee
\end{theorem}

As in Section~\ref{sec:lowweighted}, we will use  lemmas from the proof for the scalar case given in   \cite{Sc2}. 

Using \eqref{resolve}, we write
\begin{multline}\label{resolve2}
\mathcal R_V^+-\mathcal R_V^- = \\ \mathcal R_0^+-\mathcal R_0^--\mathcal R_0^+ v_1 [(h^+)^{-1}S+QD_0Q+E^+ ]v_2
	\mathcal R_0^+ 
	+\mathcal R_0^- v_1 [(h^- )^{-1}S+QD_0Q+E^- ]v_2
	\mathcal R_0^-. 
\end{multline}

First note that the contribution of the free resolvent terms   in \eqref{resolve2} to \eqref{stone3} immediately boils down to the scalar case because of \eqref{r0low2}. 
 
Note that using \eqref{matrixfree}, with $R_2(\lambda^2)=\frac{i}4 H_0^+(i\sqrt{2\mu+\lambda^2}|x-y|)$, we have
$\mathcal R_0^{\pm}(\mu+\lambda^2)=R_0^{\pm}(\lambda^2)M_{11}+R_2(\lambda^2)M_{22}$.  
Consider the contribution of `+' terms in \eqref{resolve2} with $QD_0Q$:
\begin{multline*}
	[R_0^{+} M_{11}+R_2M_{22}]v_1QD_0Qv_2[R_0^{+} M_{11}+R_2M_{22}]
	=R_0^+M_{11}v_1QD_0Qv_2M_{11}R_0^+\\
	+R_0^+M_{11}v_1QD_0Qv_2M_{22}R_2+R_2M_{22}v_1QD_0Qv_2M_{11}R_0^+
	+R_2M_{22}v_1QD_0Qv_2M_{22}R_2.
\end{multline*}
The bound for the first term is in  \cite[Lemma~16]{Sc2}, since $M_{11}v_1QD_0Qv_2M_{11}$ have the same cancellation (compare \eqref{Qorthogonality} above with (44) in \cite{Sc2}), and mapping properties as $vQD_0Qv$ in \cite{Sc2}, provided that $|a(x)|+|b(x)|\les\la x\ra^{-3/2-}$.
The last term is killed by the `+' and `-' cancellation.  For the second and third terms, we note that the `+' and `-'
cancellation says we need only consider
\begin{align*}
	(R_0^+-R_0^-)M_{11}v_1QD_0Qv_2M_{22}R_2+R_2M_{22}v_1QD_0Qv_2M_{11}(R_0^+-R_0^-).
\end{align*}
The following propositions finishes the proof of Theorem~\ref{lowprop2} for the contribution of $QD_0Q$ terms in \eqref{resolve2}. 
\begin{prop}\label{nonweightQDQ}
 If $|a(x)|+|b(x)|\les \la x\ra^{-1-}$, then
we have
$$
\sup_{x,y}\Big|\int_0^\infty e^{it\lambda^2}\lambda \chi(\lambda)  \big((R_0^+-R_0^-)M_{11}v_1QD_0Qv_2M_{22}R_2\big)(x,y)
d\lambda\Big|\les \f1t.
$$
The same bound holds for the contribution of $R_2M_{22}v_1QD_0Qv_2M_{11}(R_0^+-R_0^-)$.
\end{prop}
The following variation of stationary phase will be useful in the proof.  See Lemma~2 in \cite{Sc2}.
\begin{lemma}\label{stat phase}

	Let $\phi'(0)=0$ and $1\leq \phi'' \leq C$.  Then,
  	\begin{align*}
    		\bigg| \int_{-\infty}^{\infty} e^{it\phi(\lambda)} \mathcal E(\lambda)\, d\lambda \bigg|
    		\lesssim \int_{|\lambda|<|t|^{-\frac{1}{2}}} |\mathcal E(\lambda)|\, d\lambda
    		+|t|^{-1} \int_{|\lambda|>|t|^{-\frac{1}{2}}} \bigg( \frac{|\mathcal E(\lambda)|}{|\lambda^2|}+
    		\frac{|\mathcal E'(\lambda)|}{|\lambda|}\bigg)\, d\lambda.
  	\end{align*}

\end{lemma}

\begin{proof}[Proof of Proposition~\ref{nonweightQDQ}]
Recall that from \eqref{R2 bounds} we have
$$
|R_2(\lambda^2)(y_1,y) |, |\partial_\lambda R_2(\lambda^2)(y_1,y) |\les 1+\log^- |y_1-y|.
$$
Also recall that
\begin{multline}\label{r0+-r0-}
 R_0^+(\lambda^2)(x,x_1)-R_0^-(\lambda^2)(x,x_1)=\frac{i}{2}J_0(\lambda|x-x_1|)\\  =\rho(\lambda|x-x_1|)+ e^{i\lambda|x-x_1|}
\widetilde\chi(\lambda|x-x_1|)\omega_+(\lambda|x-x_1|)+e^{-i\lambda|x-x_1|}\widetilde\chi(\lambda|x-x_1|)
\omega_-(\lambda|x-x_1|), \\
\rho(z)=\chi(z) [1+\widetilde O_1 (z^2)],\,\,\,\,\,\,\,\,\, \omega_\pm(z) =\widetilde O\big((1+|z|)^{-\frac{1}{2}}\big).
\end{multline}
The contribution of $\rho$ is:
$$
\int_0^\infty e^{it\lambda^2 }\lambda \chi(\lambda) \rho(\lambda|x-x_1|)
\big( M_{11}v_1QD_0Qv_2M_{22}\big)(x_1,y_1) R_2(\lambda^2)(y_1,y) dx_1dy_1
d\lambda.
$$
After an integration by parts, we  can bound the $\lambda$ integral above by
\begin{multline*}
O[t^{-1}(1+\log^- |y_1-y|)]+\f1t\int_0^\infty \Big|\f{d}{d\lambda} \big(
  \chi(\lambda) \rho(\lambda|x-x_1|)
  R_2(\lambda^2)(y_1,y) \big)\Big|
d\lambda\\= O[t^{-1}(1+\log^- |y_1-y|)].
\end{multline*}
The last equality follows from the bounds on $R_2$ and $\partial_\lambda R_2$, and by noting that 
$$|\partial_\lambda \rho(\lambda|x-x_1|)|\lesssim |x-x_1|\chi_{[0,|x-x_1|^{-1}]}(\lambda).$$
This bound suffices for the contribution of $\rho$ since 
$QD_0Q$ is absolutely bounded and $\|v_2(y_1) (1+\log^-|y-y_1|)\|_{L^2_{y_1}}\les 1$. 

For the remaining terms in \eqref{r0+-r0-}, we only consider the case of $\omega_-$ and $t>0$ (the bound for $\omega_+$ follows from an integration by parts since the phase has no critical point). The corresponding $\lambda$ integral is
$$
\int_0^\infty e^{it\lambda^2-i\lambda|x-x_1|}\lambda \chi(\lambda)\widetilde\chi(\lambda|x-x_1|) \omega_-(\lambda|x-x_1|)
 R_2(\lambda^2)(y_1,y)
d\lambda.
$$
It suffices to prove that this integral is $O[t^{-1}(1+\log^- |y_1-y|)]$.

The phase, $\phi=\lambda^2-\lambda|x-x_1|/t$, has a critical point at
$\lambda_0=|x-x_1|/2t$. Let
$$
\mathcal E(\lambda)=\lambda \chi(\lambda) \omega_-(\lambda|x-x_1|)\widetilde\chi(\lambda|x-x_1|) R_2(\lambda^2)(y_1,y).
$$
By Lemma~\ref{stat phase} we estimate the $\lambda$ integral by
\begin{align} \label{alemmabound}
		\int_{|\lambda-\lambda_0|<t^{-1/2}}|\mathcal E(\lambda)|\, d\lambda+ t^{-1}
		\int_{|\lambda-\lambda_0|>t^{-1/2}} \Big( \frac{|\mathcal E(\lambda)|}{|\lambda-\lambda_0|^2}
	    +\frac{|\mathcal E'(\lambda)|}{|\lambda-\lambda_0|}\Big)\, d\lambda.
\end{align}
The first integral in \eqref{alemmabound} is bounded by
$$
(1+\log^- |y_1-y|)\int_{|\lambda-\lambda_0|<t^{-1/2}} \f{\lambda\chi(\lambda)}{(1+\lambda|x-x_1|)^{1/2}}\, d\lambda,
$$
which is $O[t^{-1}(1+\log^- |y_1-y|)]$ if $\lambda_0\les t^{-1/2}$ (by ignoring the denominator). In the case $\lambda_0\gg t^{-1/2}$ we have  $ \lambda  \sim \lambda_0$, and thus we can bound the integral by
$$
 t^{-1/2} \f{(1+\log^- |y_1-y|) \lambda_0}{(1+\lambda_0|x-x_1|)^{1/2}}\les t^{-1/2}
\f{ (1+\log^- |y_1-y|) \lambda_0^{1/2}}{|x-x_1|^{1/2}} \les t^{-1}(1+\log^- |y_1-y|).
$$
Now note that
$$
|\mathcal E^\prime(\lambda)|\les (1+\log^- |y_1-y|) \f{\widetilde \chi(\lambda |x-x_1|)}{(1+\lambda |x-x_1|)^{1/2}}.
$$
Using this, we bound the second integral in \eqref{alemmabound} by
$$
  t^{-1}(1+\log^- |y_1-y|)\int_{|\lambda-\lambda_0|>t^{-1/2}}  \frac{\widetilde \chi(\lambda |x-x_1| )}{(1+\lambda |x-x_1|)^{1/2}} \Big(\f\lambda{|\lambda-\lambda_0|^2}+\f1{|\lambda-\lambda_0|}\Big)   \, d\lambda.
$$
We have two cases: $\lambda_0\ll t^{-\frac{1}{2}}$ and $\lambda_0\gtrsim t^{-\frac{1}{2}}$. In the former case, we have $|\lambda-\lambda_0|\approx \lambda$. Thus we can bound the integral above by
\begin{multline*}
  t^{-1}(1+\log^- |y_1-y|)\int 
  \frac{\widetilde \chi(\lambda |x-x_1| )}{ (1+\lambda |x-x_1|)^{1/2}}    \, \f{d\lambda}{\lambda} \\ = 
   t^{-1}(1+\log^- |y_1-y|)\int 
  \frac{\widetilde \chi(\lambda  )}{ (1+\lambda )^{1/2}}    \, \f{d\lambda}{\lambda}\les  t^{-1}(1+\log^- |y_1-y|).
\end{multline*}
In the latter case we bound the integral by
\begin{multline*}
  t^{-1}(1+\log^- |y_1-y|)\int_{|\lambda-\lambda_0|>t^{-1/2}}  \frac{\widetilde \chi(\lambda |x-x_1| )}{  |x-x_1|^{1/2}} \Big(\f{ \lambda_0^{1/2}}{|\lambda-\lambda_0|^2}+\f1{|\lambda-\lambda_0|^{3/2}}+\f1{ \lambda^{3/2}}\Big)   \, d\lambda \\
 \les t^{-1}(1+\log^- |y_1-y|) \Big(\f{(\lambda_0 t)^{1/2}}{|x-x_1|^{1/2}}+\f{t^{1/4}}{|x-x_1|^{1/2}}+1\Big)
 \les  t^{-1}(1+\log^- |y_1-y|).
\end{multline*}
In the last inequality we used the definition of $\lambda_0$ and the assumption that $\lambda_0\gtrsim t^{-1/2}.$

\end{proof}

We now consider the
contribution of `+' terms with $S$  in \eqref{resolve2} to \eqref{stone3}:
\begin{multline*}
	[R_0^{+} M_{11}+R_2M_{22}]\frac{v_1Sv_2}{h^+}[R_0^{+} M_{11}+R_2M_{22}]
	=\f{R_0^+M_{11}v_1Sv_2M_{11}R_0^+}{h^+}\\
	+\f{R_0^+}{h^+}M_{11}v_1Sv_2M_{22}R_2+R_2M_{22}v_1Sv_2M_{11}\f{R_0^+}{h^+}
	+\f{1}{h^+}R_2M_{22}v_1QD_0Qv_2M_{22}R_2.
\end{multline*}
The bound for the first term (for the difference of '+' and '-') is in  \cite[Lemma~17]{Sc2}, it requires that $|a(x)|+|b(x)|\les\la x\ra^{-3/2-}$.  The following propositions take care of the remaining terms.
\begin{prop}\label{nonweightS1}
 If $|a(x)|+|b(x)|\les \la x\ra^{-1-}$, then
we have
$$
\int_0^\infty e^{it\lambda^2}\lambda \chi(\lambda)  \Big(\f1{h^+(\lambda)}-\f1{h^-(\lambda)}\Big)(R_2M_{22}v_1Sv_2M_{22}R_2)(x,y)
d\lambda\\
= O\Big(\f1t  \Big).
$$
\end{prop}
 \begin{prop}\label{nonweightS2}
 If $|a(x)|+|b(x)|\les \la x\ra^{-1-}$, then
we have
$$
\int_0^\infty e^{it\lambda^2}\lambda \chi(\lambda)  \Big(\frac{R_0^+}{h^+(\lambda)}-\frac{R_0^-}{h^-(\lambda)}\Big)\big(M_{11}v_1Sv_2M_{22}R_2\big)(x,y)
d\lambda\\
= O\Big(\f1t  \Big).
$$
\end{prop}
\begin{proof}[Proof of Proposition~\ref{nonweightS1}]
It suffices to prove that the $\lambda $ integral is $O[t^{-1}(1+\log^- |y_1-y|)(1+\log^- |x_1-x|)]$ as in the proof of Proposition~\ref{nonweightQDQ}.
 
Noting that
\begin{align*}
    \frac{1}{h^+(\lambda)}-\frac{1}{h^-(\lambda)} =\frac{c}{(\log\lambda+c_1)^2+c_2^2},
\end{align*}
and the bounds \eqref{R2 bounds} on $R_2$ and its derivative, it suffices to prove that 
$$
\int_0^\infty e^{it\lambda^2} \frac{\lambda \chi(\lambda)}{ (\log\lambda+c_1)^2+c_2^2} d\lambda= O(1/t).
$$
This follows by a single integration by parts.
\end{proof}
\begin{proof}[Proof of Proposition~\ref{nonweightS2}]
Using \eqref{scalarfree} we have 
\begin{multline*}
    \frac{R_0^+(\lambda^2)(x,x_1)}{h^+(\lambda)}-\frac{R_0^-(\lambda^2)(x,x_1)}{h^-(\lambda)}\\
    =iJ_0(\lambda |x-x_1|)  \Big(\frac{1}{h^+(\lambda)}+\frac{1}{h^-(\lambda)} \Big)
    -Y_0(\lambda |x-x_1|) \Big(\frac{1}{h^+(\lambda)}-\frac{1}{h^-(\lambda)} \Big)
    \\=C\frac{2i J_0(\lambda |x-x_1|))(\log\lambda+c_1)+2ic_2 Y_0(\lambda |x-x_1|)}{(\log\lambda+c_1)^2+c_2^2}.
\end{multline*}
Noting the bounds 
$$
\frac{ \log\lambda+c_1 }{(\log\lambda+c_1)^2+c_2^2} =O(1), \,\,\,\,\text{ and } \partial_\lambda \Big(\frac{ \log\lambda +c_1}{(\log\lambda+c_1)^2+c_2^2}\Big)=O(1/\lambda),
$$ 
we see that the proof for the contribution of the term containing $J_0$ follows from the proof of Proposition~\ref{nonweightQDQ},
since this term satisfies the same bounds that $J_0$ does.

Essentially the same argument works for the contribution of the $Y_0$ term. Indeed, note that $Y_0$ behaves like $J_0$ for $\lambda |x-x_1| \gtrsim 1$, and for $\lambda|x-x_1|\ll 1$, we have the following harmless dependence on $|x-x_1|$: 
$$
 \frac{\chi(\lambda)\chi(\lambda|x-x_1|)Y_0(\lambda|x-x_1|)}{(\log\lambda+c_1)^2+c_2^2} =(1+\log^-|x-x_1|)\widetilde O\big(\chi(\lambda)\chi(\lambda|x-x_1|)\big).
$$
This estimate follows from the bound
$$
|\log(\lambda|x-x_1|)|\les |\log\lambda|+\log^-|x-x_1|,\,\,\,\text{ provided } \lambda|x-x_1|\les 1.
$$

\end{proof}
 
The bound for the contribution of the error term, $E^\pm$,  in \eqref{resolve2} to \eqref{stone3} follows from \cite[Lemma~18]{Sc2} since $E^\pm$ satisfies the bounds that the lemma requires and also $\mathcal R_0$ satisfies the same bounds that $R_0$ satisfies.

\section{Proof of Theorem~\ref{thm:nonweighted} for energies separated from the thresholds}

We note \cite[Lemma~3]{Sc2}, which we modify slightly to match
the notation we have employed throughout this paper.  We  define
$$
	\|W\|_{\mathcal K}:=\sup_{x\in \R^2} \int_{\R^2}
	(1+\log^-|x-y|)^2 |W(y)|\, dy.
$$
\begin{lemma}\label{SchlagLem3}
	Let $\{1,2,\dots,m\}=J\cup J^*$ be a partition.  Then
	\begin{multline}
		\sup_{L\geq 1}\sup_{x_0,x_m\in \R^2} \int_{\R^{2(m-1)}}
		\bigg|\int_0^\infty \lambda e^{i(t\lambda^2\pm\lambda \sum_{j\in J}|x_{j+1}-x_j|)} \widetilde\chi(\lambda)
		\chi(\lambda/L)\prod_{j\in J}\rho_+(\lambda|x_{j+1}-x_j|)\\
		\prod_{\ell\in J^*}\rho_-(\lambda |x_{\ell-1}-x_\ell|)
		\,d\lambda \bigg| \prod_{k=1}^{m-1}|W(x_k)| dx_1\dots
		dx_{m-1} \les |t|^{-1}\|W\|_{\mathcal K}^{m-1}
	\end{multline}
	
\end{lemma}

In the  proof of Theorem~\ref{thm:weighted} for energies separated from the threshold,  we encountered this
integral in \eqref{high born1}. By the discussion in that proof the finite terms of the Born series in
\eqref{born series} can be written as a finite sum of terms in this form where $W$ is $\pm V_1$ or $\pm V_2$.
We note that by the decay assumptions on $V_1$ and $V_2$, we always have $\|W\|_{\mathcal K}<\infty$. Therefore Lemma~\ref{SchlagLem3} suffices to handle the contribution of the finite terms of the Born series, \eqref{born series}.

It remains to consider the contribution of the tail of the series \eqref{born series}, see \eqref{I def} and \eqref{tailcont}. 

Note that for $\lambda |x-x_1|>1$, the scalar free resolvent $R_0(\lambda^2)(x,x_1)$ has the oscillatory term $e^{\pm i\lambda |x-x_1|}$. 
If a $\lambda$ derivative hits one of the free resolvents at the edges the oscillatory term produces $|x-x_1|$  which can not be bounded uniformly in $x$.  This was not an issue in  the weighted case since we are able to allow some growth in  $x$ and $y$.

For the non-weighted case this problem is  overcome in \cite[Proposition 4]{Sc2} by changing the phase in the $\lambda$-integral 
by writing 
$$
R_0^\pm(\lambda^2)(\cdot,x)=e^{\pm i \lambda |x|} G_{\pm,x}(\lambda)(\cdot).
$$
Note that oscillatory part changes the phase in the integral and $G_{\pm,x}(\lambda)$ and its derivatives does not grow in $x$ since
differentiating  $G_{\pm,x}(\lambda)$ in $\lambda$ produces $|x-x_1|-|x|=O(|x_1|)$ (which can be killed by the decay assumption on the potential). In  \cite[Proposition 4]{Sc2}, this implies the required bound by an application of stationary phase and by using limiting absorption principle.

Since $\mathcal R_0^\pm$ satisfies the limiting absorption principle with the same weights, it suffices to see that we can define and bound the functions $G_{\pm,x}(\lambda)$ analogously. Let 
$$
\mathcal R_0^\pm(\mu+\lambda^2)(\cdot,x)=e^{\pm i \lambda |x|} \mathcal G_{\pm,x}(\lambda)(\cdot),
$$
where
$$
\mathcal G_{\pm,x}(\lambda)(x_1)= G_{\pm,x}(\lambda)(x_1)M_{11}+ e^{\mp i \lambda |x|} R_2(\lambda^2)(x_1,x) M_{22}.
$$
It suffices to consider the second summand. Using the definition of $R_2$ we have 
$$
e^{\mp i \lambda |x|} R_2(\lambda^2)(x_1,x)=e^{ \pm i\lambda (|x|-|x-x_1|)} \rho(\sqrt{2\mu+\lambda^2}|x-x_1|) 
e^{(\pm i\lambda-\sqrt{2\mu+\lambda^2})|x-x_1|},
$$
where $\rho(u)=\widetilde O(\log(u))$ for $u\in[0,1/2]$ and $\rho(u)= \widetilde O(u^{-1/2})$ for $u>1/2$. 
We note  that (see the proof of \cite[Proposition 4]{Sc2}), modulo the second exponential factor, this is identical to $G_{\pm,x}(\lambda)(x_1)$. 
Therefore the required bounds follow by noting that
$$
\partial_\lambda^k e^{(\pm i\lambda-\sqrt{2\mu+\lambda^2})|x-x_1|} = O(1),\,\,\,\,\,\,\, k=0,1,2,...
$$

\begin{large}
\noindent
{\bf Acknowledgment. \\}
\end{large}
The first author was partially supported by National Science Foundation grant DMS-1201872.  The second author acknowledges
the support of an AMS Simons Travel Grant.

\end{document}